\documentclass[11pt]{amsart}
\usepackage{amsmath,amssymb,color,cite,colonequals,pdflscape,verbatim,mathrsfs}

\usepackage{tikz}
\usetikzlibrary{positioning}

\usepackage{tikz-cd}
\numberwithin{equation}{section}
\newtheorem{thm}[equation]{Theorem}
\newtheorem{prop}[equation]{Proposition}
\newtheorem{lemma}[equation]{Lemma}
\newtheorem{cor}[equation]{Corollary}

\theoremstyle{definition}
\newtheorem{rmk}[equation]{Remark}

\newtheorem{definition}[equation]{Definition}

\newtheorem{prop-def}{Proposition-Definition}

\newcommand{\bP}{\mathbb{P}}

\DeclareMathOperator{\CH}{CH}

\newcommand{\M}{\mathcal{M}}

\DeclareMathOperator{\codim}{codim}

\DeclareMathOperator{\Aut}{Aut}

\DeclareMathOperator{\Spec}{Spec}

\makeatletter
\def\udots{\mathinner{\mkern1mu\raise\p@
\vbox{\kern7\p@\hbox{.}}\mkern2mu
\raise4\p@\hbox{.}\mkern2mu\raise7\p@\hbox{.}\mkern1mu}}
\makeatother

\usepackage[colorlinks,pagebackref,pdftex, bookmarks=false]{hyperref}
\hypersetup{citecolor=blue}

\renewcommand{\bar}[1]{#1\llap{$\overline{\phantom{\rm#1}}$}}
\usepackage[colorlinks,pagebackref,pdftex, bookmarks=false]{hyperref}

\newif\ifpdf
\ifx\pdfoutput\undefined
\pdffalse
\else
\pdfoutput=1
\pdftrue
\fi

\ifpdf
\usepackage[colorlinks,pagebackref,pdftex, bookmarks=false]{hyperref}
\else
\usepackage[colorlinks,pagebackref]{hyperref}
\fi
\begin{document}
\title{Chow rings of moduli spaces of genus $0$ curves with collisions}
\author{William C. Newman}
\begin{abstract}
Introduced in~\cite{BB}, simplicially stable spaces are alternative compactifications of $\mathcal{M}_{g,n}$ generalizing Hassett's moduli spaces of weighted stable curves. We give presentations of the Chow rings of these spaces in genus $0$ using techniques developed by the author in~\cite{Me1}. When considering the special case of $\bar\M_{0,n}$, this gives a new proof of Keel's presentation of $\operatorname{CH}(\bar\M_{0,n})$.
\end{abstract}
\begin{comment}
    Introduced in [BB], simplicially stable spaces are alternative compactifications of $\mathcal{M}_{g,n}$ generalizing Hassett's moduli spaces of weighted stable curves. We give presentations of the Chow rings of these spaces in genus $0$ using techniques developed by the author in [New25]. When considering the special case of $\bar\mathcal{M}_{0,n}$, this gives a new proof of Keel's presentation of $\operatorname{CH}(\bar{\mathcal{M}}_{0,n})$.
\end{comment}
\maketitle
\section{Introduction}
The simplicially stable spaces were introduced by Blankers and Bozlee in~\cite{BB}.  They are compactifications of $\M_{g,S}$, the moduli space of smooth genus $g$ curves with markings in a finite set $S$, which are indexed by (abstract) simplicial complexes $\mathcal K$ and are denoted $\bar\M_{g,\mathcal K}$. The simplicial complex controls whether the marked points may ``collide'' in $\bar\M_{g,\mathcal K}$ (see Definition~\ref{def} for a precise definition). For example, given distinct $i,j\in S$, if $\{i,j\}\in \mathcal K$ and $(C,\{p_i: i\in S\})\in \bar\M_{g,\mathcal K}$, then we may have $p_i=p_j$, in contrast with the usual compactification of stable curves, $\bar\M_{g,S}$. 

Simplicially stable spaces generalize the earlier moduli spaces of weighted stable curves first studied by Hassett in~\cite{Hassett}. Given weight data $\mathcal A=\{a_i:i\in S\}$, one gets a simplicial complex $\mathcal K_{\mathcal A}:=\{T\subseteq S| \sum_{i\in T} a_i \leq 1\}$, and one has $\bar\M_{g,\mathcal A}=\bar\M_{g,\mathcal K_{\mathcal A}}$.

Much recent work has focused on the Chow rings of moduli spaces of stable curves \cite{BDL2,CL2,DLPV,ELarson}. We study the Chow rings of these simplicially stable spaces in genus $0$. Notably, these spaces give all stable (in the sense of Smyth~\cite{Smyth}) smooth modular compactifications of $\M_{0,n}$ (combine~\cite[Theorem 7.8]{MSvAX} with~\cite[Lemma 4.15]{BB}). Previously, Keel gave a presentation for the Chow ring of $\bar\M_{0,S}$ in~\cite{Keel}. This was generalized by Kannan, Karp, and Li in~\cite{KKL}, where the Chow rings of so-called ``heavy/light'' Hassett spaces were calculated in genus $0$ using tropical methods.

Before stating the main theorem, we recall Keel's presentation of the Chow ring of $\bar\M_{0,n}$. Given a subset $I\subseteq [n]$ so that $2\leq \# I\leq n-2$, let $D_I$ be the closure of the set of curves in $\bar\M_{0,n}$ with two components, with markings from $I$ on one component and markings from the complement $I^c$ on the other. 
\begin{thm}\!\!\!\!\emph{\cite[Theorem 1]{Keel}}
    The Chow ring of $\bar\M_{0,n}$ is generated by the classes $[D_I]$ modulo only the relations
    \begin{align*}
        [D_I]=[D_{I^c}]&,\\
        [D_I][D_J]=0 &\text{ if }I\not\subseteq J,J^c\text{ and }J,J^c\not\subseteq I, \text{ and }\\
        \sum_{\substack{i,j\in I\\ k,\ell \not\in I}} [D_I]=\sum_{\substack{i,k\in I\\ j,\ell \not\in I}} [D_I]& \text{ for }i,j,k,\ell\in [n]\text{ distinct (WDVV relations).}
    \end{align*}
    Moreover, the cycle class map for $\bar\M_{0,n}$ from the Chow ring to the cohomology ring is an isomorphism.
\end{thm}
Our main theorem is the following.
\begin{thm}[Theorem~\ref{main2}]\label{main}
    The Chow ring of $\bar\M_{0,\mathcal K}$ is generated by the classes of boundary divisors $D$, and the only relations are 
    \begin{align*}
        &\prod_{i=1}^d [D_i]=0\text{ for any set $\{D_i\}$ with }\bigcap_{i=1}^d D_i=\emptyset,\\
        &[D]\cdot [D']=[D]\cdot [D'']\text{ if }D\cap D'=D\cap D''\text{, and}\\
        &\text{the pushforward of the WDVV relations from } \bar\M_{0,S}.
    \end{align*}
     Moreover, the cycle class map for $\bar\M_{0,\mathcal K}$ from the Chow ring to the cohomology ring is an isomorphism.
\end{thm}
Another version of this theorem is given in Theorem~\ref{main3}. We recover Keel's theorem from this in Section~\ref{5.1}. Our proof does not generalize Keel's, so we obtain a new, short proof of Keel's result based just on the general theory of the stratification of the boundary, a Chow-K\"unneth property, and some use of higher Chow groups. Keel's original approach cannot work in the setting of an arbitrary $\bar\M_{0,\mathcal K}$: Keel describes $\bar\M_{0,n}$ as an iterated blow-up of $(\bP^1)^{n-3}$, and the spaces $\bar\M_{0,\mathcal K}$ are not necessarily projective~\cite[Example 11.1]{MSvAX}.

In~\cite{Me1}, the author gave new proofs of presentations of the (integer coefficient) Chow rings $\CH(\bar\M_{1,n})$ for $n\leq 4$ using higher Chow groups and the motivic K\"unneth property. The same techniques are used in this paper, giving a much shorter display of their usefulness.

There is another description of $\CH(\bar\M_{0,n})$ given by Kontsevich and Manin in~\cite{KM94,KM96}. This description is additive; it describes all linear relations between the classes of the strata. The main theorem in the paper~\cite{BS2}, which greatly inspired the author's work in~\cite{Me1} and hence in the present paper, was an analogue/generalization of this description to $\CH(\mathfrak M_{0,n}),$ where $\mathfrak M_{0,n}$ is the stack of prestable curves. Such a description of $\CH(\bar\M_{0,\mathcal K})$ should be true and provable following the proof given in~\cite{BS2}. This Kontsevich-Manin description is known for Hassett spaces by Ceyhan~\cite{Ceyhan}.

The paper is organized as follows: in Section 2, we give background on simplicially stable spaces and the techniques developed in~\cite{Me1}; in Section 3, we discuss the stratification of $\bar\M_{g,\mathcal K}$ by dual graph; in Section 4, the main theorem is proven; and in Section 5, we look at specializing the main theorem to some specific simplicial complexes $\mathcal K$.

\subsection*{Acknowledgments}  
The author thanks Sebastian Bozlee for help with understanding simplicially stable spaces. The author also thanks Siddarth Kannan for helpful discussions about previous results. Finally, the author thanks an anonymous referee for their helpful suggestions. 

This material is based upon work supported by the National Science Foundation under Grant No. DMS-2231565.

\subsection*{AI Disclosure}
The author asked Claude Sonnet 5.5 about the correctness of an earlier draft. Claude found an error in the proof and statement of Proposition~\ref{kernel}, and said it did not think this would affect the main theorem. The author realized it did, and found the second example in Section~\ref{necessity}. Claude then found another error in the proof of Proposition~\ref{kernel}, noticed that it affected the main theorem, and suggested the first example in Section~\ref{necessity}. Claude also noticed an error in the previous definition of a generic $(\Gamma,\Gamma')$-structure. Claude and Gemini were used to find typos.

\subsection*{Notation} 
\begin{itemize}
    \item We work over a field of arbitrary characteristic. All varieties are defined over this field. The spectrum of this field is denoted $\operatorname{pt}$.
    \item For a variety $X$, $\CH(X)$ denotes the total Chow group, i.e.\ it is the direct sum of cycles of all dimensions modulo rational equivalence. We will not need to refer to the grading. Similarly, $\CH(X,j)$ is the total $j$-th higher Chow group.
    \item The set $\{1,2,\dots,n\}$ is denoted $[n]$. 
\end{itemize}
\section{Background}

\subsection{Simplicially Stable Spaces}
We recall the definition of simplicially stable spaces from~\cite{BB}. Let $g\geq 0$ and $\mathcal K$ be a simplicial complex on the underlying set $S$. If $g=0$, we require that $\mathcal K$ be \emph{at least triparted}, meaning every partition of $S$ into elements of $\mathcal K$ has at least $3$ parts. 
\begin{definition}\label{def}
    The simplicially stable space $\bar\M_{g,\mathcal K}$ parameterizes nodal genus $g$ curves with smooth marked points indexed by $S$, $(C,\{p_i|i\in S\})$, with finitely many automorphisms, so that 
    \begin{itemize}
        \item for each $x\in C$, the set $\{i\in S| p_i=x\}$ is in $\mathcal K$ and
        \item for each rational tail  $Z\subseteq C$ (i.e.\ a closed connected subcurve of arithmetic genus $0$ with $Z\cap \bar{C\setminus Z}$ a single point), the set $\{i\in S| p_i\in Z\}$ is not in $\mathcal K$.\footnote{One need only check this condition for \emph{irreducible} rational tails $Z\subseteq C$.} 
    \end{itemize}
\end{definition}
In general, $\bar\M_{g,\mathcal K}$ is a smooth proper Deligne-Mumford stack~\cite[Theorem 4.18]{BB}. In the genus $0$ case, by~\cite[Proposition 11.5]{MSvAX}, we have even more:
\begin{thm}
    $\bar\M_{0,\mathcal K}$ is a smooth proper variety. 
\end{thm}
By~\cite[Example 11.1]{MSvAX}, it is not true that $\bar\M_{0,\mathcal K}$ is always projective, though the Hassett spaces are projective~\cite[Theorem 2.1]{Hassett}. 

Lastly, if $C$ is a nodal curve and $Z\subseteq C$ is a rational tail, we can \emph{contract} $Z$ to get a curve $C'=\bar{C\setminus Z}$ and a morphism $C\to C'$ that is an isomorphism away from $Z$ and is constant on $Z$. If $C$ has markings from a set $S$, the morphism $C\to C'$ induces markings from $S$ on $C'$. For disjoint rational tails $Z_1,Z_2\subseteq C$, contracting in either order produces the same curve. The following theorem says that we can do these contractions in families to get a morphism of moduli spaces.

\begin{thm}\label{morphism}\!\!\!\!\emph{\cite[Theorem 6.3]{BB}} We have a birational morphism 
\[\pi_\mathcal K: \bar\M_{0,S}\to \bar\M_{0,\mathcal K}\]
which is the identity on $\M_{0,S}$ and more generally sends a pointed curve $C$ to the contraction of all maximal rational tails $Z\subseteq C$ with $\{i\in S|p_i\in Z\}\in \mathcal K$. 
\end{thm}
Because $\pi_{\mathcal K}$ is birational, for any $\alpha \in \CH(\bar\M_{0,\mathcal K})$, the projection formula \cite[Example 8.1.7]{Fulton} gives
\[\pi_{\mathcal K*}(\pi_{\mathcal K}^*(\alpha))=\pi_{\mathcal K*}(\pi_{\mathcal K}^*(\alpha)\cdot 1)=\alpha\cdot \pi_{\mathcal K*}(1)=\alpha,\]
hence the pullback
$\pi^*_{\mathcal K}: \CH(\bar\M_{0,\mathcal K})\to \CH(\bar\M_{0,S})$
is injective. However, not much can be gleaned from this. For example, it is not clear from this injection that $\CH(\bar\M_{0,\mathcal K})$ is generated by the classes of divisors.

\subsection{Higher Chow Groups}
The proof of Theorem~\ref{main} uses some basic properties of first higher Chow groups. These are useful because they give an extension of the usual localization exact sequence for Chow groups.
\begin{prop}
    Suppose $X$ is a variety and $Z\subseteq X$ is a closed subvariety. If $U\subseteq X$ is the open complement of $Z$, there is a long exact sequence
    \[\dots\to \CH(Z,p)\xrightarrow{\iota_*} \CH(X,p)\xrightarrow{j^*}\CH(U,p)\xrightarrow{\partial_p}\CH(Z,p-1)\to\dots\]
    extending the localization exact sequence.
\end{prop}
We will only need to consider the following portion of this long exact sequence:
\[\CH(U,1)\xrightarrow{\partial_1} \CH(Z)\to\CH(X)\to \CH(U)\to 0.\]
We also use the following fact, which is~\cite[Lemma 2.5]{Me1}.
\begin{lemma}\label{Module}
Suppose $S$ is a smooth, irreducible variety. Then for any morphism $f: X\to S$, we get a $\CH(S)$-module structure on $\CH(X,p)$, compatible with pushforwards, pullbacks, and connecting homomorphisms in the localization exact sequence. 
\end{lemma}

\subsection{Motivic K\"unneth Property}
We recall the basic properties of the motivic K\"unneth property (MKP) defined in~\cite[Section 2.4]{Me1}.  

\begin{prop}\label{MKPprops}
    \mbox{}
    \begin{enumerate}
        \item Suppose $X$ is a variety with closed subvariety $Z$ and set $U:=X\setminus Z$. If two out of three of $X,U$, and $Z$ have the MKP, then so does the third.
        \item If $X$ and $X'$ have the MKP, so does $X\times X'$.
        \item If $X$ has the MKP, then $X$ has the Chow-K\"unneth property, i.e.\ for all varieties $Y$, the map
        \[\CH(X)\otimes\CH(Y)\to \CH(X\times Y)\]
        is an isomorphism.
        \item $\M_{0,S}$ has the MKP.
        \item  Let $X$ be a smooth proper variety with the MKP. For any reasonable cohomology theory $H$ (e.g.\ singular or \'etale), the cycle class map
    \[\CH(X)\otimes_{\mathbb Z} A\to H(X)\]
    is an isomorphism, where $A=H(\operatorname{pt})$.
    \end{enumerate}
\end{prop}

(1), (2), and (3) are proven in~\cite[Proposition 2.17]{Me1}. (4) is proven in~\cite[Proposition 7.2]{Me1}. (5) is proven for singular cohomology (with rational coefficients) in~\cite[Lemma 3.11]{CL2}. This proof can be adapted for any cohomology theory admitting cycle class maps satisfying standard assumptions.  Alternatively, by~\cite[Theorem 4.1]{Totaro}, the Chow motive of $X$ is a direct sum of Tate twists of the motive of $\operatorname{pt}$. Because any reasonable cohomology theory factors through the category of Chow motives, the result follows.

\begin{rmk}
    The use of the MKP in this paper can be replaced by the easier-to-define Chow-K\"unneth generation property. See~\cite[Remark 2.26]{Me1}.
\end{rmk}

\section{Boundary Stratification}
Let $g\geq 0$ and $\mathcal K$ be a simplicial complex on a finite set $S$ which is at least triparted if $g=0$. We define the boundary strata on $\bar\M_{g,\mathcal K}$ analogously to the boundary strata on $\bar\M_{g,n}$ from the appendix of~\cite{GP} (see also~\cite[Section 2]{BS1}). 

\begin{definition}
    A $\mathcal K$-stable graph $\Gamma$ consists of 
    \begin{itemize}
        \item a vertex set $V(\Gamma)$, with a function $g:V(\Gamma)\to \mathbb Z_{\geq 0}$ assigning a genus,
        \item a set of half edges $H(\Gamma)$, each attached to a vertex according to a function $f:H(\Gamma)\to V(\Gamma)$,
        \item an involution $\iota$ on $H(\Gamma)$, whose size-$2$ orbits on $H(\Gamma)$ define the (full) edges $E(\Gamma)$ and whose fixed points are called legs, and
        \item a surjective function $h: S\to H(\Gamma)^{\iota}$ from the set of labels $S$ to the fixed points of $\iota$,
    \end{itemize}
    such that
    \begin{enumerate}
        \item the graph $(V(\Gamma),E(\Gamma))$ is connected, 
        \item $\{i\in S| h(i)=t\}\in \mathcal K$ for each $t\in H(\Gamma)^\iota$, 
        \item $\{i\in S| h(i)\in H(v)\}\notin \mathcal K$ for each leaf $v\in V(\Gamma)$ (i.e.\ a vertex with only one full edge) with $g(v)=0$, where $H(v):=f^{-1}(v)$ is the set of half edges at $v$, and 
        \item $2g(v)-2+n(v)> 0$ for all $v\in V(\Gamma)$, where $n(v):=\#H(v)$. 
    \end{enumerate}
\end{definition}
The differences between this definition and the usual definition for $\bar\M_{g,n}$ are that the function $h$ need not be injective and that we need the extra conditions (2) and (3). Thus, we may think of a $\mathcal K$-stable graph as a usual stable graph where certain sets of markings can live on the same leg.

Given a geometric point of $\bar\M_{g,\mathcal K}$, one can obtain a $\mathcal K$-stable graph whose vertices correspond to the irreducible components, edges correspond to the intersections of components, and legs correspond to the set of points which are marked.
\begin{definition}
    Given a $\mathcal K$-stable graph $\Gamma$ and a vertex $v\in V(\Gamma)$, let $\mathcal K(v)$ be the simplicial complex on the set $H(v)$ containing singletons and subsets $A\subseteq H(v)^{\iota}\subseteq H(v)$ such that $\{i\in S|h(i)\in A\}\in \mathcal K$.
\end{definition} 
\begin{lemma}
    Let $\Gamma$ be a $\mathcal K$-stable graph and $v\in V(\Gamma)$ with $g(v)=0$. Then $\mathcal K(v)$ is at least triparted.
\end{lemma}
\begin{proof}
Suppose there are at least two (full) edges coming out of $v$. By stability, there must be a third half edge at $v$, either corresponding to a leg or another full edge. The only elements of $\mathcal K(v)$ containing each of the half edges at $v$ corresponding to the two full edges are the singletons, so any partition of $H(v)$ must have at least three parts. 

Next, suppose there is only one full edge coming out of $v$. Again, the only element of $\mathcal K(v)$ containing the half edge corresponding to this full edge is the singleton. If $\mathcal K(v)$ were not at least triparted, we would have that the set of all other half edges at $v$ is an element of $\mathcal K(v)$. By definition, this says that the set of $i\in S$ so that $h(i)\in H(v)$ is an element of $\mathcal K$, but this contradicts condition (3) of the definition of a $\mathcal K$-stable graph.  

Finally, suppose that there are zero full edges coming out of $v$. Because $\mathcal K$-stable graphs are connected, $v$ is the only vertex in $V(\Gamma)$. As $g(v)=0$, we also have $g=0$, and so $\mathcal K$ is at least triparted. Any partition of $H(v)$ into $A_1,\dots,A_m\in \mathcal K(v)$ induces a partition of $S$ by the sets
\[\widetilde A_p:=\{i\in S| h(i)\in A_p\}.\]
By definition of $\mathcal K(v)$, the sets $\widetilde A_p$ are in $\mathcal K$. Because $\mathcal K$ is at least triparted, so is $\mathcal K(v)$.
\end{proof}

This lemma allows us to make the following definition.
\begin{definition}   
    Let $\Gamma$ be a $\mathcal K$-stable graph. Define
    \[\bar\M_\Gamma:=\prod_{v\in V(\Gamma)}\bar\M_{g(v),\mathcal K(v)}\]
    and 
    \[\M_\Gamma:=\prod_{v\in V(\Gamma)}\M_{g(v),H(v)}\subseteq \bar\M_\Gamma.\]
    We have a morphism
    \[\xi_\Gamma: \bar\M_\Gamma \to \bar\M_{g,\mathcal K}\]
    defined by gluing the curves along the marked points corresponding to edges. Here 
    \[g=\# E(\Gamma)-\#V(\Gamma)+1+\sum_{v\in V(\Gamma)} g(v).\]
    Let $\bar\M^\Gamma:=\xi_\Gamma(\bar\M_\Gamma)$. 
    
   Define the partial order $\leq$ on $\mathcal K$-stable graphs by declaring $\Gamma'\leq \Gamma$ if $\Gamma'$ can be obtained by replacing the vertices $v\in V(\Gamma)$ by $\mathcal K(v)$-stable graphs $\Gamma'_v$. A choice of $\{\Gamma'_v\}_{v\in V(\Gamma)}$ is called a \emph{$\Gamma$-structure} on $\Gamma'$; we say that $\Gamma'$ is a \emph{specialization} of $\Gamma$. Finally, set
    \[\M^\Gamma:=\bar\M^\Gamma\setminus \left(\bigcup_{\Gamma': \Gamma'<\Gamma} \bar\M^{\Gamma'} \right).\]
\end{definition}
\noindent Note that we leave the $\mathcal K$ implicit in our notation for the above notions.
\begin{lemma}\label{leq}
    For $\mathcal K$-stable graphs $\Gamma,\Gamma'$, $\Gamma'\leq \Gamma$ if and only if $\bar\M^{\Gamma'}\subseteq \bar\M^\Gamma$.
\end{lemma}
\begin{proof}
    Every geometric point of $\bar\M^\Gamma$ is equal to $\xi_\Gamma$ applied to a geometric point of $\bar\M_\Gamma$. A geometric point of $\bar\M_\Gamma$ is a tuple $(C_v)_{v\in V(\Gamma)}$ of $\mathcal K(v)$-stable curves. If $\Gamma'_v$ is the $\mathcal K(v)$-stable graph of $C_v$, then the stable graph of $\xi_\Gamma((C_v)_{v})$ is obtained by replacing the vertices $v\in V(\Gamma)$ by $\Gamma_v'$. Hence, the $\mathcal K$-stable graph $\Gamma'$ of any geometric point in $\bar\M^\Gamma$ satisfies $\Gamma'\leq \Gamma$. Then, if $\bar\M^{\Gamma'}\subseteq\bar\M^\Gamma$, because some points in $\bar\M^{\Gamma'}$ have $\Gamma'$ as their stable graph, $\Gamma'\leq \Gamma$. Conversely, if $\Gamma'\leq \Gamma$, the map $\xi_{\Gamma'}:\bar\M_{\Gamma'}\to \bar\M^{\Gamma'}\subseteq \bar\M_{g,\mathcal K}$ factors through $\xi_{\Gamma}:\bar\M_{\Gamma}\to \bar\M^\Gamma\subseteq\bar\M_{g,\mathcal K}$ by gluing together the curves in the graphs $\Gamma'_v$ for $v\in V(\Gamma)$. This gives $\bar\M^{\Gamma'}\subseteq \bar\M^\Gamma$.
\end{proof}
\begin{prop}\label{strata}
The collection of locally closed subschemes $\{\M^\Gamma\}_{\Gamma}$ of $\bar\M_{g,\mathcal K}$ stratifies $\bar\M_{g,\mathcal K}$ and $\xi_\Gamma(\M_\Gamma)=\M^\Gamma$.    
\end{prop}
\begin{proof}
    We first claim that the geometric points of $\M^\Gamma$ are exactly those with $\mathcal K$-stable graph $\Gamma$. If a geometric point $(C,\{p_i:i\in S\})$ has $\mathcal K$-stable graph $\Gamma$, then, by construction, $(C,\{p_i:i\in S\})$ can be glued together from curves represented by the vertices of $\Gamma$ along the edges. Thus, $(C,\{p_i:i\in S\})$ is in $\xi_\Gamma(\bar\M_\Gamma)=\bar\M^\Gamma$. If $(C,\{p_i:i\in S\})\in \bar\M^{\Gamma'}\subseteq\bar\M^\Gamma$, we have that $\Gamma\leq \Gamma'\leq \Gamma$, so $\Gamma=\Gamma'$. Thus, $(C,\{p_i:i\in S\})\in \M^\Gamma$. Conversely, suppose $(C,\{p_i:i\in S\})\in \M^\Gamma$.  If $\Gamma'$ is the $\mathcal K$-stable graph of $(C,\{p_i:i\in S\})$, we have $\Gamma'\leq \Gamma$. But since $(C,\{p_i:i\in S\})\notin \bar\M^{\Gamma''}$ for all $\Gamma''< \Gamma$, we have $\Gamma=\Gamma'$. 

    The first claim follows because this description of $\M^\Gamma$ shows they are disjoint and their union is the entire space. The second claim follows because it is clear that a geometric point has $\mathcal K$-stable graph $\Gamma$ if and only if it is in $\xi_{\Gamma}(\M_\Gamma)$. 
\end{proof}

\begin{prop}\label{special}
    Suppose $g=0$. Then, for any $\mathcal K$-stable graph $\Gamma$, $\xi_\Gamma$ is a closed embedding.
\end{prop}
\begin{rmk}
    This proposition is not true if $g>0$. In the usual $\bar\M_{g,n}$ case, there is a related result for arbitrary genus which says that $\xi_\Gamma$ induces a locally closed embedding $[\M_\Gamma/\Aut(\Gamma)]\to \bar\M_{g,n}$ for stable graphs $\Gamma$~\cite[Proposition 2.5]{BS1}. This is true in our setting, by essentially the same proof given in~\cite{BS1}, but we do not need this.
\end{rmk}

This proposition says that $\xi_\Gamma$ identifies $\bar\M_\Gamma$ with $\bar\M^\Gamma$. In the next section, we will only refer to this space as $\bar\M_\Gamma$.
\begin{proof}
Given a simplicial complex $\mathcal J$ on a set $T$, say a $\mathcal J$-stable graph $\Theta$ is \emph{simple} if either
\begin{enumerate}
    \item $\Theta$ has two vertices and the map $h: T\to H(\Theta)^\iota$ is a bijection, or
    \item $\Theta$ has one vertex, and $\#T=1+\# H(\Theta)^{\iota}$.
\end{enumerate}
(We will see below in Corollary~\ref{codim} that simple graphs are exactly those such that the codimension of $\bar\M^\Theta$ in $\bar\M_{0,\mathcal J}$ is $1$.)

Given $\Gamma$, we can find a sequence $\Gamma=\Gamma_n\leq \dots\leq \Gamma_0$, where $\Gamma_0$ is the stable graph of a point in $\M_{0,S}\subseteq \bar\M_{0,\mathcal K}$, such that $\Gamma_{i+1}$ is obtained by replacing a vertex $v_i\in V(\Gamma_i)$ with a simple $\mathcal K(v_i)$-stable graph. We get a factorization of $\xi_\Gamma$ as
\[\bar\M_\Gamma=\bar\M_{\Gamma_n}\xrightarrow{\xi_{n}}\bar\M_{\Gamma_{n-1}}\xrightarrow{\xi_{n-1}} \dots\xrightarrow{\xi_{1}} \bar\M_{\Gamma_0}=\bar\M_{0,\mathcal K}.\]
Suppose we knew the proposition for all simple graphs. Then each of the $\xi_i$ is a closed embedding, because they can be written as the product of $\xi_\Theta$, for $\Theta$ a simple graph, and the identity on $\prod_{v\neq v_{i+1}}\bar\M_{0,\mathcal K(v)}$. It follows that $\xi_\Gamma$ is a closed embedding, because it is the composition of closed embeddings. Thus, it suffices to show the proposition for simple graphs $\Gamma$.

To show that $\xi_\Gamma$ is a closed embedding for simple $\Gamma$, we may assume the base field is algebraically closed, by faithfully flat descent. Then, if we know that $\xi_\Gamma$ is injective on $\Spec A$-points for local Artin algebras $A$ over the base field, we get $\xi_\Gamma$ is injective on closed points and tangent vectors. By~\cite[\href{https://stacks.math.columbia.edu/tag/0E8M}{Tag 0E8M}]{stacks-project}, $\xi_\Gamma$ is then a closed embedding.

If $\Gamma$ has one vertex, then $\xi_\Gamma$ is clearly injective on all $T$-points for any scheme $T$, because it simply duplicates a section. Now assume $\Gamma$ has two vertices, $v_1$ and $v_2$. Because $\Gamma$ is simple with two vertices, $h$ induces a bijection $S\to H(v_1)^\iota\amalg H(v_2)^\iota$. Let $*_i\in H(v_i)$ be the half edges not fixed by $\iota$, so that $\{*_1,*_2\}$ is the unique edge in $\Gamma$. 

We first consider the case when $A=K$ is a field. Suppose we have pointed curves $C_1,D_1\in \bar\M_{0,\mathcal K(v_1)}(\Spec K)$ and $C_2,D_2\in \bar\M_{0,\mathcal K(v_2)}(\Spec K)$ and an isomorphism $\varphi: C\to D$, where $C=\xi_\Gamma(C_1,C_2)$ and $D=\xi_\Gamma(D_1,D_2)$. For a $\mathcal J$-stable curve $B,$ let $\Theta_B$ denote its stable graph.  Then $\varphi$ induces an isomorphism $\varphi:\Theta_C\to \Theta_D$ of $\mathcal K$-stable graphs. If we can show that $\varphi$ sends $\Theta_{C_i}$ to $\Theta_{D_i}$, this says that $\varphi(C_i)=D_i$, making $(C_1,C_2)=(D_1,D_2)$ as $K$-points. To show this, we claim it suffices to show that $\varphi$ sends the edge $e_C:=\{*_1,*_2\}\in E(\Theta_C)$ to the edge $e_D:=\{*_1,*_2\}\in E(\Theta_D)$. This is because, using that $\Theta_C$ is a tree, we get either $\varphi(\Theta_{C_i})=\Theta_{D_i}$ or $\varphi(\Theta_{C_i})=\Theta_{D_{3-i}}$, and the latter possibility is ruled out because $\varphi$ preserves marked points. 

Cutting $\Theta_D$ at an edge $e\in E(\Theta_D)$ yields a partition of the legs of $\Theta_D$ of size $2$ by considering which of the two connected components the leg is on. Cutting $\Theta_D$ at either $\varphi(e_C)$ or $e_D$ gives the same induced partition of legs, because $\varphi$ preserves legs. By stability, every vertex has at least three half edges; if a vertex has a leg attached, it must be on the same side of the cut, and otherwise has an edge which eventually leads to some legs, and therefore must also be on the same side of the cut. Therefore, there is only one way to cut $\Theta_D$ at an edge that gives a fixed partition of the legs, and so $\varphi(e_C)=e_D$.

Finally, for a general local Artin ring $A$, let $K$ be the residue field. An isomorphism $\varphi:C\to D$ induces an isomorphism $\varphi_K:C_K\to D_K$, which must send $(C_i)_K$ to $(D_i)_K$ by the above. Because $A$ is a local Artin ring, for any finite type $X/\Spec A$, we have that the scheme map $X_K\to X$ is a homeomorphism, and so we can conclude that $C_i$ must map to $D_i$ under $\varphi$. This gives $(C_1,C_2)=(D_1,D_2)$ as $A$-points.
\end{proof}

\begin{cor}\label{codim}
    The codimension of $\bar\M^\Gamma$ in $\bar\M_{0,\mathcal K}$ is equal to $\# E(\Gamma)+\# S-\# H(\Gamma)^\iota$.
\end{cor}
\begin{proof}
We know $\dim \bar\M_{0,\mathcal K}=\dim \M_{0,S}=\# S-3$. Moreover, using that $\Gamma$ is a tree, we have
\begin{align*}
    \dim \bar\M^\Gamma&=\dim \bar\M_\Gamma\\
    &=\dim\left(\prod_{v\in V(\Gamma)} \bar\M_{0,\mathcal K(v)}\right)\\
    &=\sum_{v\in V(\Gamma)} \dim \M_{0,H(v)}\\
    &=\sum_{v\in V(\Gamma)} \#H(v)-3\\\displaybreak
    &=\#H(\Gamma)-3\#V(\Gamma)\\
    &=\#H(\Gamma)^\iota+2\#E(\Gamma)-3\#V(\Gamma)\\
    &=\#H(\Gamma)^\iota+2\#E(\Gamma)-3(\#E(\Gamma)+1)\\
    &=\#H(\Gamma)^\iota-\#E(\Gamma)-3,
\end{align*}
and hence the codimension is $\#E(\Gamma)+\#S-\#H(\Gamma)^\iota$.
\end{proof}

\begin{cor}\label{MKP}
    $\bar\M_{0,\mathcal K}$ has the MKP.
\end{cor}
By Proposition~\ref{MKPprops}(5), cycle class maps are isomorphisms for $\bar\M_{0,\mathcal K}$, proving the latter part of Theorem~\ref{main}. 
\begin{proof}
    By Proposition~\ref{strata} and repeated use of Proposition~\ref{MKPprops}(1), it suffices to show that $\M^\Gamma$ has the MKP for all genus $0$ $\mathcal K$-stable graphs $\Gamma$. Then, by Proposition~\ref{strata} and Proposition~\ref{special}, we have $\M^\Gamma\cong \M_\Gamma$, which is a product of $\M_{0,H(v)}$, and therefore has the MKP by Proposition~\ref{MKPprops}(2,4).
\end{proof}

We say the codimension of a $\mathcal K$-stable graph $\Gamma$ is the codimension of $\bar\M^\Gamma$ in $\bar\M_{g,\mathcal K}$. By Corollary~\ref{codim}, a genus $0$ codimension one $\mathcal K$-stable graph $\Delta$ must have either one edge or precisely two markings at the same leg, i.e.\ it must look like
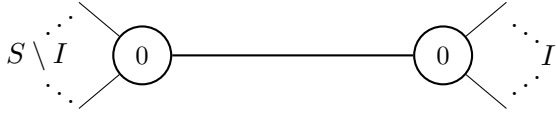
\[\begin{tikzpicture}[
  vertex/.style={
    circle, draw, thick, minimum size=2em, inner sep=0pt, font=\small
  }
]

  %--- Vertices ---
  \node[vertex] (v0) at (0,0) {$0$};
  \node[vertex] (vg) at (4,0) {$0$};

  %--- Connecting edge ---
  \draw[thick] (v0) -- (vg);

  %--- Half‑edges on the genus‑0 vertex ---
  % Label "i"
  \draw (v0.140) -- ++(140:0.7) node[above] {};
  % Label "n+1"
  \draw (v0.-140) -- ++(-140:0.7) node[below] {};
  \node at (-1.1,.4) {$\udots$};
  \node at (-1.1,-.4) {$\ddots$};
  \node at (-1.4,0) {$S\setminus I$};

  %--- Half‑edges on the genus‑g vertex ---
  % Two unlabeled half‑edges
  
  \draw (vg.40)  -- ++(40:0.7)  node[above] {};
  \node[] at (5.1,.4) {$\ddots$};

  \node[] at (5.4,0) {$I$};
  % Dots for “…”
   \node[] at (5.1,-.4) {$\udots$};
  
  % One more unlabeled half‑edge
  \draw (vg.-40) -- ++(-40:0.7)  node[below] {};

\end{tikzpicture}\]
for some subset $I\subseteq S$ or
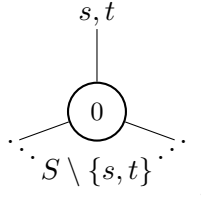
\[\begin{tikzpicture}[
  vertex/.style={
    circle, draw, thick, minimum size=2em, inner sep=0pt, font=\small
  }
]

  %--- Vertices ---
  \node[vertex] (v0) at (0,0) {$0$};

    \draw (v0.90) -- ++(90:0.7) node[above] {};
  \draw (v0.-20) -- ++(-20:0.7) node[above] {};
  % Label "n+1"
  \draw (v0.-160) -- ++(-160:0.7) node[below] {};
  \node at (1,-.5) {$\udots$};
  \node at (-1,-.4) {$\ddots$};
  \node at (0,-.8) {$S\setminus \{s,t\}$};
  \node at (0,1.3) {$s,t$};

\end{tikzpicture},\]
for some distinct $s,t\in S$. 

We end this section with a discussion of the analogue of Proposition 8 of the appendix of~\cite{GP}. We will need the following notion.
\begin{definition}
    Let $\Gamma,\Gamma'$ be two $\mathcal K$-stable graphs. A $(\Gamma,\Gamma')$-structure on a $\mathcal K$-stable graph $\Omega$ is simply a $\Gamma$-structure and a $\Gamma'$-structure on $\Omega$. Such a $(\Gamma,\Gamma')$-structure on $\Omega$ is \emph{generic} if it is not a nontrivial specialization of another $(\Gamma,\Gamma')$-graph. 
\end{definition}
\begin{thm}
    Let $\Gamma,\Gamma'$ be $\mathcal K$-stable graphs. Then we have
        \[\bar\M_\Gamma \times_{\overline\M_{g,\mathcal K}} \bar\M_{\Gamma'}=\coprod_{\substack{\Omega\text{ generic}\\ (\Gamma,\Gamma')\text{-graph}}} \bar\M_{\Omega}.\]
\end{thm}
We refer to~\cite{GP} for the proof, as there are no substantial differences. From this result, we will mainly use the following corollary:
\begin{cor}\label{tautologicalcalculus}
    Let $\Delta,\Lambda$ be genus $0$ $\mathcal K$-stable graphs, with $\codim(\Delta)=1$. Then, either 
    \begin{itemize}
        \item $\bar\M^\Lambda \subseteq \bar\M^{\Delta}$,
        \item $\bar\M^\Delta\cap \bar\M^{\Lambda}=\emptyset$, or
        \item $\bar\M^\Delta\cap \bar\M^{\Lambda}=\bar\M^\Omega$ for a $\mathcal K$-stable graph $\Omega$.
    \end{itemize}
\end{cor}
\begin{proof}
    Assume that the first two cases do not happen. Recall we have $\bar\M_\Gamma\cong \bar\M^\Gamma$, so we can apply the theorem to compute the intersection. If $\Delta$ has two vertices, then any generic $(\Delta,\Lambda)$-graph must be obtained by replacing a vertex of $\Lambda$ with an edge, placing the correct markings on their respective sides of the edge, and there is only one way to do this. If $\Delta$ has one vertex, colliding markings $i,j\in S$, then any generic $(\Delta,\Lambda)$-graph must be obtained by combining the labels on the half edges labeled $i$ and $j$, and there is only one way to do this.  
\end{proof}
We note that, because the $\bar\M^\Gamma$ are smooth for genus $0$ $\mathcal K$-stable graphs by Proposition~\ref{special}, we can use this corollary to compute products in the Chow ring of $\bar\M_{0,\mathcal K}$.
 
\section{Proof of Theorem~\ref{main}}
Let $\mathcal K$ be a simplicial complex on $S$ that is at least triparted. 

\begin{definition}
    Given distinct $i,j,k,\ell\in S$, we say that a $\mathcal K$-stable graph $\Gamma$ of codimension one \emph{separates} $\{i,j\}$ from $\{k,\ell\}$, written $ij\Gamma k\ell$, if one of the following conditions holds:
    \begin{itemize}
        \item $i$ and $j$ are at the same half edge.
        \item $k$ and $\ell$ are at the same half edge.
        \item There exists an edge in $\Gamma$ whose removal leaves two connected components, one with $i$ and $j$ and the other with $k$ and $\ell$.
    \end{itemize}
\end{definition}
Recall the WDVV relations on $\bar\M_{0,S}$:
\[\sum_{\substack{i,j\in I\\ k,\ell \not\in I}} [D_I]=\sum_{\substack{i,k\in I\\ j,\ell \not\in I}} [D_I].\]
Here $D_I$ is the subscheme $\bar\M_\Gamma$, where $\Gamma$ is the stable graph with two vertices with markings $I$ on one vertex and markings $I^c$ on the other.

We follow the notation of~\cite{Me1}: given a $\mathcal K$-stable graph $\Gamma$, always a capital Greek letter, the lower case $\gamma$ denotes the class $[\bar\M_\Gamma]$ inside the Chow group of any space containing it, such as $\bar\M_\Gamma$, $\partial\bar\M_{0,\mathcal K}$, and $\bar\M_{0,\mathcal K}$. With this notation, we can write the relations on $\bar\M_{0,S}$ as
\[\sum_{\substack{\text{codim $1$ }\Delta\\ ij\Delta k\ell}} \delta=\sum_{\substack{\text{codim $1$ }\Delta\\ ik\Delta j\ell}} \delta.\]
Recall the morphism
\[\pi_{\mathcal K}: \bar\M_{0,S}\to \bar\M_{0,\mathcal K}\]
from Theorem~\ref{morphism}. For a geometric point $C\in \bar\M_{0,S}$, the stable graph of $C$ separates $\{i,j\}$ from $\{k,\ell\}$ if and only if the $\mathcal K$-stable graph of $\pi_{\mathcal K}(C)$ separates $\{i,j\}$ from $\{k,\ell\}$. Thus, the pushforward of the WDVV relations to $\CH(\bar\M_{0,\mathcal K})$ can also be written as 
\[\sum_{\substack{\text{codim $1$ }\Delta\\ ij\Delta k\ell}} \delta=\sum_{\substack{\text{codim $1$ }\Delta\\ ik\Delta j\ell}} \delta.\]

We introduce the following shorthand notations: 
\begin{itemize}
    \item We write $\Delta_1\wedge\ldots\wedge \Delta_d=\emptyset$ as shorthand for $\bar\M_{\Delta_1}\cap \ldots\cap \bar\M_{\Delta_d}=\emptyset$, and 
    \item we write $\Delta_1\wedge \Delta_2=\Delta_1\wedge \Delta_3$ as shorthand for $\bar\M_{\Delta_1}\cap\bar\M_{\Delta_2}=\bar\M_{\Delta_1}\cap\bar\M_{\Delta_3}$ when such an intersection is nonempty of codimension $2$.
\end{itemize}
With this notation the main theorem can be written in the following way:
\begin{thm}\label{main2}
    The Chow ring of $\bar\M_{0,\mathcal K}$ is generated by the classes $\delta$ for codimension one $\mathcal K$-stable graphs $\Delta$, subject to only the relations
    \begin{align*}
        \delta_1\cdot\ldots\cdot \delta_d=0 & \text{ if } \Delta_1\wedge\ldots\wedge\Delta_d=\emptyset,\\
        \delta_1\delta_2=\delta_1\delta_3 & \text{ if $\Delta_1\wedge \Delta_2=\Delta_1\wedge \Delta_3$, and }\\
        \sum_{\substack{\text{codim $1$ }\Delta\\ ij\Delta k\ell}} {\delta}=\sum_{\substack{\text{codim $1$ }\Delta\\ ik\Delta j\ell}} {\delta} & \text{ for distinct } i,j,k,\ell\in S.
    \end{align*}
\end{thm}
We described above why the third type of relations hold. The first and second sets clearly hold as well. 

It is not difficult to describe which codimension one $\mathcal K$-stable graphs satisfy $\Delta_1\wedge\ldots\wedge \Delta_d=\emptyset$. Note that by Corollary~\ref{tautologicalcalculus}, it suffices to describe when $\Delta\wedge \Lambda=\emptyset$, where $\Delta$ has codimension one and $\Lambda$ has arbitrary codimension. By analyzing the two different cases of codimension one $\mathcal K$-stable graphs, one sees that $\Delta\wedge \Lambda=\emptyset$ if and only if 
\begin{enumerate}
    \item $\Delta$ partitions $S$ into two sets incompatibly with the partition of $S$ coming from $\Lambda$ (this is the usual condition for $\bar\M_{0,n}$ given in Keel's presentation);
    \item$\Lambda$ collides two markings of $S$ on different components of $\Delta$ or vice versa; or
    \item $\Delta$ collides two markings whose corresponding sets of collided points on $\Lambda$ do not union to an element of $\mathcal K$. 
\end{enumerate}
By saying that the partitions of $\Delta$ and $\Lambda$ are incompatible, we mean that there exists some edge $e\in E(\Lambda)$, inducing a partition of $S$ into two sets based on which side of the edge the marking lands on, so that no containments hold between the sets involved in this partition and the partition coming from $\Delta$.

\begin{definition}
    Define the ring $R_{\mathcal K}$ to be the free commutative $\mathbb Z$-algebra generated by symbols $\tilde{\delta}$ for genus $0$ $\mathcal K$-stable graphs $\Delta$ of codimension one, modulo the relations
    \begin{align*}
        \tilde\delta_1\cdot\ldots\cdot \tilde\delta_{d}=0 & \text{ if } \Delta_1\wedge\ldots\wedge\Delta_d=\emptyset,\\
        \tilde\delta_1\tilde\delta_2=\tilde\delta_1\tilde\delta_3 & \text{ if $\Delta_1\wedge \Delta_2=\Delta_1\wedge \Delta_3$, and }\\
        \sum_{\substack{\text{codim $1$ }\Delta\\ ij\Delta k\ell}} \tilde{\delta}=\sum_{\substack{\text{codim $1$ }\Delta\\ ik\Delta j\ell}} \tilde{\delta} & \text{ for distinct } i,j,k,\ell\in S.
    \end{align*}
    By the above, we have a ring homomorphism 
\[h_{\mathcal K}: R_{\mathcal K}\to \CH(\bar\M_{0,\mathcal K})\]
\[\tilde{\delta}\mapsto \delta.\]
We see that Theorem~\ref{main2} is equivalent to saying that $h_{\mathcal K}$ is an isomorphism.
\end{definition}

Now, we work toward proving the theorem. By Lemma~\ref{Module}, we can view the (higher) Chow groups of the boundary and boundary strata as $\CH(\bar\M_{0,\mathcal K})$-modules. Via the ring homomorphism $h_{\mathcal K}$, we view all of these groups as $R_{\mathcal K}$-modules.

We start with the localization exact sequence for $\partial\bar\M_{0,\mathcal K}\subseteq \bar\M_{0,\mathcal K}$:
\[\CH(\M_{0,S},1)\xrightarrow{\partial_1} \CH(\partial\bar\M_{0,\mathcal K})\to \CH(\bar\M_{0,\mathcal K})\to \CH(\M_{0,S})\to 0.\]
We first compute the image of $\partial_1$. 
\begin{prop}\label{higher}
    The image of $\partial_1$ is generated by the WDVV relations on $\bar\M_{0,\mathcal K}$:
\[\sum_{\substack{\text{codim $1$ }\Delta\\ ij\Delta k\ell}} \delta-\sum_{\substack{\text{codim $1$ }\Delta\\ ik\Delta j\ell}} \delta.\]
\end{prop}
\begin{proof}
    Recall the morphism $\pi_{\mathcal K}:\bar\M_{0,S}\to \bar\M_{0,\mathcal K}$ from Theorem~\ref{morphism}.  Note that it sends $\partial\bar\M_{0,S}$ to $\partial\bar\M_{0,\mathcal K}$. By the commutative diagram of localization exact sequences
    \begin{center}
        \begin{tikzcd}
        \CH(\M_{0,S},1)\arrow[r,"\partial_1"]\arrow[d,"\operatorname{id}"] & \CH(\partial\bar\M_{0,S})\arrow[d,"\pi_{\mathcal K*}"] \\
        \CH(\M_{0,S},1)\arrow[r,"\partial_1"] & \CH(\partial\bar\M_{0,\mathcal K}),
        \end{tikzcd}
    \end{center}
    we see that the image of the bottom $\partial_1$ is equal to the pushforward along $\pi_{\mathcal K}$ of the image of the top $\partial_1$. By~\cite[Theorem 7.8]{Me1} or~\cite[Proposition 2.36]{BS2}, this image is generated by the WDVV relations on $\bar\M_{0,S}$, so the result follows.
\end{proof}

Now, to prove Theorem~\ref{main2}, we induct on the size of $S$. If $\# S=3$, then $\bar\M_{0,\mathcal K}=\bar\M_{0,3}=\operatorname{pt}$ and $R_{\mathcal K}=\mathbb Z$, so $h_{\mathcal K}$ is an isomorphism. Next, suppose $\# S=n>3$ and the theorem holds for all (at least triparted) simplicial complexes with $\leq n-1$ elements. 

\begin{prop}\label{kernel}
    Let $\Delta$ be a codimension one $\mathcal K$-stable graph. As an $R_{\mathcal K}$-module, we have that $\CH(\bar\M_\Delta)$ is generated by $\delta=[\bar\M_{\Delta}]$ subject to the relations
    \begin{align*}
        \tilde \gamma_1\cdot\ldots\cdot\tilde\gamma_d\cdot \delta=0&\text{, if }\Gamma_1\wedge\ldots\wedge\Gamma_d\wedge\Delta=\emptyset\text{, and}\\
    \tilde \gamma_1\cdot \delta = \tilde \gamma_2\cdot \delta &\text{, if }\Gamma_1\wedge \Delta=\Gamma_2\wedge \Delta.
    \end{align*}
\end{prop}
\begin{proof}
Let $I_{\Delta}\subseteq R_{\mathcal K}$ be the ideal generated by the elements $\tilde \gamma_1\cdot\ldots\cdot\tilde\gamma_d$ for all codimension one $\{\Gamma_i\}_{i=1}^d$ with $\Gamma_1\wedge\ldots\wedge\Gamma_d\wedge\Delta=\emptyset$ and $\tilde\gamma_1-\tilde\gamma_2$ for all codimension one $\Gamma_1,\Gamma_2$ with $\Gamma_1\wedge\Delta=\Gamma_2\wedge\Delta$.
Each of these elements gets mapped to $0$ by \[R_{\mathcal K}\xrightarrow{h_{\mathcal K}}\CH(\bar\M_{0,\mathcal K})\xrightarrow{\xi_\Delta^*} \CH(\bar\M_\Delta),\] so we get an induced map
\[\bar\xi_{\Delta}^*: {R_{\mathcal K}}/{I_\Delta}\to \CH(\bar\M_\Delta).\]
The proposition is then equivalent to saying that $\bar\xi_{\Delta}^*$ is an isomorphism. We show this by finding $\mathbb Z$-algebra generating sets for both sides, showing that $\bar\xi_{\Delta}^*$ is a bijection on these generators, and showing that all relations on $\CH(\bar\M_{\Delta})$ are satisfied on $R_{\mathcal K}/I_\Delta.$ This is sufficient, as being bijective on generators implies $\bar\xi_\Delta^*$ is surjective, and the latter condition implies that the kernel is $0$.

The ring $R_{\mathcal K}$ is generated (as a $\mathbb Z$-algebra) by the classes $\tilde\gamma$ for genus $0$ $\mathcal K$-stable graphs $\Gamma$ of codimension one. Using the WDVV relations, we see that the generator $\tilde\delta$ is not necessary. Hence, we have that $R_{\mathcal K}/I_\Delta$ is generated by the classes $\tilde\gamma$ for genus $0$ $\mathcal K$-stable graphs $\Gamma$ of codimension one  not equal to $\Delta$ with $\Gamma\wedge \Delta\neq \emptyset$. To understand generators and relations of the $\mathbb Z$-algebra $\CH(\bar\M_\Delta)$, we use the isomorphism
\begin{equation}\label{identify}
    \bigotimes_{v\in V(\Delta)}R_{\mathcal K(v)}\xrightarrow{\otimes h_{\mathcal K(v)}}\bigotimes_{v\in V(\Delta)} \CH(\bar\M_{0,\mathcal K(v)})=\CH(\bar\M_\Delta),
\end{equation}
where the first map is an isomorphism by the inductive hypothesis, and the equality holds by Proposition~\ref{MKPprops}(3) and Corollary~\ref{MKP}; recall that if $A_i$ is a $\mathbb Z$-algebra with generators $\{a_{ij}\}_j$ and relations $\{r_{ik}\}_k$ for $i=1,2$, then $\{a_{1j}\otimes 1\}_j\cup \{1\otimes a_{2j}\}_j$ and $\{r_{1k}\otimes 1\}_k\cup \{1\otimes r_{2k}\}_k$ are sets of $\mathbb Z$-algebra generators and relations, respectively. 

Suppose first that $\Delta$ has two vertices $v_1$ and $v_2$. Then \eqref{identify} implies that $\CH(\bar\M_\Delta)$ is generated by elements of the form $\gamma'\otimes 1$, where $\Gamma'$ is a codimension one $\mathcal{K}(v_1)$-stable graph, and $1\otimes \gamma'$, where $\Gamma'$ is a codimension one $\mathcal{K}(v_2)$-stable graph. Let $\Gamma\neq \Delta$ be a codimension one $\mathcal K$-stable graph so that $\Gamma\wedge \Delta\neq \emptyset$, and let $\Omega$ be the graph so that $\bar\M_\Omega=\bar\M_\Delta\cap \bar\M_\Gamma$, which exists by Corollary~\ref{tautologicalcalculus}. Again by Corollary~\ref{tautologicalcalculus}, we have 
\[\bar\xi^*_{\Delta}(\tilde{\gamma})=\begin{cases}
    \gamma' \otimes 1 & \text{if }\Omega\text{ is obtained by gluing }\Gamma'\text{ into }\Delta\text{ at }v_1\\
    1 \otimes \gamma' & \text{if }\Omega\text{ is obtained by gluing }\Gamma'\text{ into }\Delta\text{ at }v_2 
\end{cases}\]
under the identification given in \eqref{identify}.  Given a $\mathcal K(v_1)$-stable graph $\Lambda'$ of arbitrary codimension, if $\Omega$ is the graph obtained by gluing $\Lambda'$ into $\Delta$ at $v_1$, and $\Lambda$ is the contraction of $\Omega$ by the edge connecting $v_1$ and $v_2$, then $\xi_\Delta^{-1}(\bar\M_{\Lambda})=\bar\M_{\Lambda'}\times \bar\M_{0,\mathcal K(v_2)}$, and $\Lambda$ is the unique graph of codimension equal to that of $\Lambda'$ with this property. Hence, $\bar\xi_\Delta^*$ is bijective on the generating sets. 

The relations between the generators of $\CH(\bar\M_\Delta)$ are
\begin{align*}
    (\gamma_1'\otimes 1)\cdot \ldots\cdot (\gamma_d'\otimes 1)&=0 \text{ if } \Gamma_1'\wedge\dots\wedge\Gamma_d'=\emptyset\\
    (\gamma_1'\otimes 1)\cdot (\gamma_2'\otimes 1)&= (\gamma_1'\otimes 1)\cdot (\gamma_3'\otimes 1) \text{ if }\Gamma_1'\wedge\Gamma_2'=\Gamma_1'\wedge\Gamma_3'\\
    \sum_{\substack{\text{codim $1$ }\Gamma'\\ ij\Gamma' k\ell}} \gamma'\otimes 1&=\sum_{\substack{\text{codim $1$ }\Gamma'\\ ik\Gamma' j\ell}} \gamma'\otimes 1 \text{ for distinct } i,j,k,\ell\in H(v_1),
\end{align*}
\noindent and their analogues coming from $\mathcal K(v_2)$; by symmetry, it suffices to show only that these sets of  relations hold on $R_{\mathcal K}/I_\Delta$. Suppose $\Gamma_1',\ldots,\Gamma_d'$ are codimension one $\mathcal K(v_1)$-stable graphs with $\Gamma_1'\wedge\dots\wedge\Gamma_d'=\emptyset$. Let $\Gamma_i$ be the unique codimension one $\mathcal K$-stable graph so that $\xi_{\Delta}^{-1}(\bar\M_{\Gamma_i})=\bar\M_{\Gamma_i'}\times \bar\M_{0,\mathcal K(v_2)}$. Then $\bar\M_{\Gamma_1}\cap \dots\cap \bar\M_{\Gamma_d}\cap \bar\M_{\Delta}=(\bar\M_{\Gamma_1'}\cap \dots\cap \bar\M_{\Gamma_d'})\times \bar\M_{0,\mathcal K(v_2)}=\emptyset$. Hence, $\tilde\gamma_1\cdot\ldots\cdot\tilde\gamma_d\in I_\Delta$.

Next, consider the relation on $\CH(\bar\M_\Delta)$ of the second type given by $\Gamma_1',\Gamma_2'$ and $\Gamma_3'$. We may assume that the codimension of $\bar\M_{\Gamma_1'}\cap \bar\M_{\Gamma_2'}$ is two. Let $\Gamma_i$ be the unique codimension one $\mathcal K$-stable graph so that $\xi^{-1}_{\Delta}(\bar\M_{\Gamma_i})=\bar\M_{\Gamma_i'}\times \bar\M_{0,\mathcal K(v_2)}$. Define $\Gamma_{ij}$ so that $\bar\M_{\Gamma_{ij}}=\bar\M_{\Gamma_i}\cap \bar\M_{\Gamma_j}$, and define $\Gamma_{ij}'$ analogously. Then 
\[\xi_{\Delta}^{-1}(\bar\M_{\Gamma_{12}})=\xi_{\Delta}^{-1}(\bar\M_{\Gamma_{13}})=\bar\M_{\Gamma_{12}'}\times \bar\M_{0,\mathcal K(v_2)},\]
so $\Gamma_{12}=\Gamma_{13}$ by uniqueness. Thus, $\tilde\gamma_1\cdot \tilde\gamma_2=\tilde\gamma_1\cdot \tilde\gamma_3$ in $R_{\mathcal K}$, and hence in $R_{\mathcal K}/I_\Delta$.

Finally, consider the relation on $\CH(\bar\M_\Delta)$ of the third type given by distinct $i,j,k,\ell\in H(v_1)$:
\[\sum_{\substack{\text{codim $1$ }\Gamma'\\ ij\Gamma' k\ell}} \gamma'\otimes 1=\sum_{\substack{\text{codim $1$ }\Gamma'\\ ik\Gamma' j\ell}} \gamma'\otimes 1.\]
If none of $i,j,k,\ell$ correspond to the half edge connecting $v_1$ to $v_2$, set $i^*:=i$. Otherwise, if $i$ corresponds to this half edge, take $i^*\in H(v_2)$ not corresponding to the edge connecting $v_2$ to $v_1$. Then, using that $h: S\to H(\Delta)^\iota$ is a bijection to identify $i^*,j,k,\ell$ with elements of $S$, we can consider the relation
\[\sum_{\substack{\text{codim $1$ }\Gamma\\ i^*j\Gamma k\ell}} \tilde{\gamma}=\sum_{\substack{\text{codim $1$ }\Gamma\\ i^*k\Gamma j\ell}} \tilde{\gamma},\]
on $R_{\mathcal K}/I_\Delta$. Because the markings $j,k,\ell$ are attached to the vertex $v_1$, all $\Gamma$ appearing in the above expression must have $\Omega$ obtainable by attaching a graph onto $v_1$ and hence give $\bar\xi_\Delta^*(\tilde{\gamma})=\gamma'\otimes 1$, where $\Omega$ is as above: $\bar\M_{\Omega}=\bar\M_\Delta\cap \bar\M_\Gamma$. If $i=i^*$, it is clear that $ij\Gamma k\ell$ if and only if $ij\Gamma' k\ell$, so we see this relation maps by $\bar\xi_\Delta^*$ to the desired one. Otherwise, because $i^*$ is attached to $v_2$, we have that $i^*j\Gamma k\ell$ if and only if $ij\Gamma'k\ell$, so this relation maps by $\bar\xi_\Delta^*$ to the desired one. Thus, the relations of $\CH(\bar\M_\Delta)$ of the third type hold on $R_{\mathcal K}/I_\Delta$. We conclude $\bar\xi_\Delta^*$ is an isomorphism in this case.

Now suppose that $\Delta$ has one vertex, $v$, and say $h(s)=h(t)$ for distinct $s,t\in S$. Then $\CH(\bar\M_\Delta)$ is generated as a $\mathbb Z$-algebra by $\gamma'$ for codimension one genus $0$ $\mathcal{K}(v)$-stable graphs $\Gamma'$. $H(v)$ is identified with the set $S/\{s=t\}$ via $h$, and $\mathcal K(v)$-stable graphs are identified with $\mathcal K$-stable graphs with $h(s)=h(t)$. Fix a codimension one graph $\Gamma\neq \Delta$ so that $\Delta\wedge \Gamma\neq \emptyset$, and let $\Omega$ be the graph so that $\bar\M_\Omega=\bar\M_\Delta\cap \bar\M_\Gamma$. Then $\Omega$ is just the graph $\Gamma$ with the two half edges labeled $s$ and $t$ identified. Then we have
\[\bar\xi_{\Delta}^*(\tilde{\gamma})=\gamma',\]
where $\Gamma'$ is the $\mathcal K(v)$-stable graph associated with the $\mathcal K$-stable graph $\Omega$. Given a $\mathcal K(v)$-stable graph $\Lambda'$ of arbitrary codimension, we can find a $\mathcal K$-stable graph $\Lambda$ with $\xi_{\Delta}^{-1}(\bar\M_\Lambda)=\bar\M_{\Lambda'}$ by removing the $t$ label from the half edge it is on and adding a new half edge with label $t$. Such a $\Lambda$ is unique among $\mathcal K$-stable graphs of the same codimension as $\Lambda'$, unless there are other markings on the half edge of $\Lambda'$ labeled by $\{s,t\}$. When the codimension of $\Lambda'$ is one, the only other choice is obtained by removing the label $s$ instead. These two choices of $\Lambda$ have equal classes in $R_{\mathcal K}/I_\Delta$ by the relations of $I_\Delta$, so $\bar\xi_{\Delta}^*$ is a bijection on generators.

The relations between the set of generators $\{\gamma'\}$ of $\CH(\bar\M_\Delta)$ are
\begin{align*}
    \gamma'_1\cdot \ldots \cdot \gamma'_d&=0\text{ if }\Gamma_1'\wedge\dots\wedge \Gamma_d'=\emptyset\\
    \gamma_1'\gamma_2'&=\gamma_1'\gamma_3' \text{ if }\Gamma_1'\wedge\Gamma_2'=\Gamma_1'\wedge\Gamma_3'\\
    \sum_{\substack{\text{codim $1$ }\Gamma'\\ ij\Gamma' k\ell}} \gamma'&=\sum_{\substack{\text{codim $1$ }\Gamma'\\ ik\Gamma' j\ell}} \gamma' \text{ for distinct } i,j,k,\ell\in S/\{s=t\}.
\end{align*}
To show injectivity of $\bar\xi_\Delta^*$, we need to show that all of these relations hold when pulled back to $R_{\mathcal K}/I_\Delta$. Relations of the first kind can be seen to hold in a completely analogous manner to the case when $\Delta$ has two vertices.

Next, consider a relation of the second kind given by $\Gamma_1',\Gamma_2'$ and $\Gamma_3'$. We may assume that the codimension of $\bar\M_{\Gamma_1'}\cap \bar\M_{\Gamma_2'}$ is $2$. If $\Gamma_1'$ has two vertices, we can use the case for $\Delta$ having two vertices to conclude $\Gamma_2'=\Gamma_3'$, and so there is nothing to show. Otherwise $\Gamma_1'$ has one vertex. If $\Gamma_2'$ has two vertices, then the fact that $\Gamma_1'$ has one vertex and $\Gamma_1'\wedge\Gamma_2'=\Gamma_1'\wedge \Gamma_3'$ implies that $\Gamma_2'=\Gamma_3'$, and so there is nothing to show. Thus, we are reduced to the case where each $\Gamma_i'$ has one vertex; we can also assume $\Gamma_2'\neq \Gamma_3'$. Then we can say $\Gamma_1'$ collides markings $a,b$, $\Gamma_2'$ collides markings $a,c$, and $\Gamma_3'$ collides markings $b,c$. We can take lifts $\Gamma_i$ to $\mathcal K$-stable graphs. If one of the markings $a,b,c$ is equal to $\bar s=\bar t\in S/\{s=t\}$, we choose the lift with marking $s$ being collided and marking $t$ on its own half edge. Then, one still has $\Gamma_1\wedge \Gamma_2=\Gamma_1\wedge \Gamma_3$, hence $\tilde\gamma_1\tilde\gamma_2=\tilde\gamma_1\tilde\gamma_3$.

Finally, consider the relation of the third type given by distinct $i,j,k,\ell\in S/\{s=t\}$:
\[\sum_{\substack{\text{codim $1$ }\Gamma'\\ ij\Gamma' k\ell}} \gamma'=\sum_{\substack{\text{codim $1$ }\Gamma'\\ ik\Gamma' j\ell}} \gamma'.\]
By abuse of notation, let $i,j,k,\ell$ also denote lifts to $S$. We have the relation
\[\sum_{\substack{\text{codim $1$ }\Gamma\\ ij\Gamma k\ell}} \tilde{\gamma}=\sum_{\substack{\text{codim $1$ }\Gamma\\ ik\Gamma j\ell}} \tilde{\gamma}\]
on $R_{\mathcal K}/I_\Delta$. It is clear that $\Gamma$ separates $\{i,j\}$ from $\{k,\ell\}$ if and only if $\Gamma'$ separates $\{i,j\}$ from $\{k,\ell\}$, where $\Gamma'$ is so that $\bar\xi_\Delta^*(\tilde{\gamma})=\gamma'$. Thus, this relation maps to the desired one, and so $\bar\xi_{\Delta}^*$ is an isomorphism.
\end{proof}

\begin{cor}\label{surjective}
     Suppose $\Delta$ is a $\mathcal K$-stable graph of codimension one, $\Theta$ is a $\mathcal K$-stable graph of codimension $2$ and $\bar\M_\Theta\subseteq \bar\M_\Delta$. Then the pullback $\CH(\bar\M_\Delta)\to \CH(\bar\M_\Theta)$ is surjective.  
\end{cor}
\begin{proof}
Because $\bar\M_\Theta\subseteq \bar\M_\Delta$, we have that $\Theta$ is obtained by gluing graphs $\Theta_{v}$ at the vertices $v\in V(\Delta)$. By considering the codimensions, we have that there is some vertex $v_0\in V(\Delta)$ with $\Theta_{v_0}$ a codimension one graph. If there is another vertex, its graph has codimension $0$. Proposition~\ref{kernel} implies that
\[R_{\mathcal K(v_0)}\to \CH(\bar\M_{0,\mathcal K(v_0)})\to \CH(\bar\M_{\Theta_{v_0}})\]
is surjective, so the latter map is surjective. Using Proposition~\ref{MKPprops}(3) and Corollary~\ref{MKP}, we can write the pullback as
\[\CH(\bar\M_\Delta)=\bigotimes_{v\in V(\Delta)} \CH(\bar\M_{0,\mathcal K(v)})\to \bigotimes_{v\in V(\Delta)} \CH(\bar\M_{\Theta_v})=\CH(\bar\M_\Theta).\]
This homomorphism of tensor products is induced by the identity for $v\neq v_0$ and the surjection $\CH(\bar\M_{0,\mathcal K(v_0)})\to \CH(\bar\M_{\Theta_{v_0}})$ for $v=v_0$, hence the pullback is surjective.
\end{proof}

\begin{proof}[Proof of Theorem~\ref{main2}]
    First, we compute an $R_{\mathcal K}$-module presentation for $\CH(\partial\bar\M_{0,\mathcal K})$. Let $\xi_{\Delta\Delta'}^{\Delta}$ denote the inclusion of $\bar\M_\Delta\cap \bar\M_{\Delta'}$ into $\bar\M_\Delta$. Writing $\partial \bar\M_{0,\mathcal K}=\bigcup\limits_{\text{codim $1$ } \Delta} \bar\M_\Delta,$ the ``union of closed subschemes'' exact sequence for Chow groups \cite[Example 1.3.1c]{Fulton} gives that
    \[\bigoplus_{\text{codim $1$ }  \Delta,\Delta'} \CH(\bar\M_\Delta\cap \bar\M_{\Delta'})\xrightarrow{a} \bigoplus_{\text{codim $1$ } \Delta} \CH(\bar\M_{\Delta})\xrightarrow{b} \CH(\partial\bar\M_{0,\mathcal K})\to 0\]
    is exact, where 
    \[a=\bigoplus_{\text{codim $1$ }  \Delta,\Delta'} \xi^{\Delta}_{\Delta\Delta'*}-\xi^{\Delta'}_{\Delta\Delta'*}\]
    and
    \[b=\bigoplus_{\text{codim $1$ } \Delta} \xi_{\Delta*}.\]
    If $\bar\M_\Delta\cap \bar\M_{\Delta'}\neq \emptyset$, this intersection is equal to $\bar\M_{\Theta}$ for some codimension $2$ $\mathcal K$-stable graph $\Theta$ by Corollary~\ref{tautologicalcalculus}. By Corollary~\ref{surjective}, the pullback $\CH(\bar\M_\Delta)\to \CH(\bar\M_{\Theta})$ is surjective. Hence, $\CH(\bar\M_{\Theta})$ is generated by $\theta$ as an $R_{\mathcal K}$-module. Thus, we have that the image of 
    \[\xi^{\Delta}_{\Delta\Delta'*}-\xi^{\Delta'}_{\Delta\Delta'*}: \CH(\bar\M_{\Theta})\to \CH(\bar\M_{\Delta})\oplus\CH(\bar\M_{\Delta'})\]
    is generated by 
    \[\xi_{\Delta\Delta'*}^{\Delta}(\theta)-\xi_{\Delta\Delta'*}^{\Delta'}(\theta)=\tilde{\delta}'\delta-\tilde{\delta}\delta'\]
    as an $R_{\mathcal K}$-module. Note that these relations also hold for $\Delta,\Delta'$ with $\Delta\wedge \Delta'=\emptyset$, because both terms are $0$ in this case. By the presentation of $\CH(\bar\M_{\Delta})$ as an $R_{\mathcal K}$-module from Proposition~\ref{kernel}, the exact sequence gives an $R_{\mathcal K}$-module presentation for $\CH(\partial\bar\M_{0,\mathcal K})$ as the free $R_{\mathcal K}$-module on $\delta$ for codimension one graphs $\Delta$, modulo the relations
    \begin{align*}
        \tilde \gamma_1\cdot\ldots\cdot\tilde\gamma_d\cdot \delta,&\text{ for all $\{\Gamma_i\}_{i=1}^d,\Delta$ with }\Gamma_1\wedge\ldots\wedge\Gamma_d\wedge\Delta=\emptyset\\
    \tilde \gamma_1\delta-\tilde \gamma_2 \delta, &\text{ for all $\Gamma_1,\Gamma_2,\Delta$ with }\Gamma_1\wedge\Delta=\Gamma_2\wedge\Delta\\
         \tilde{\delta}'\delta-\tilde{\delta}\delta', &\text{ for all $\Delta,\Delta'$}.
    \end{align*}

     By Proposition~\ref{higher}, we have an exact sequence
   \[0\to\frac{\CH(\partial\bar\M_{0,\mathcal K}
   )}{\text{WDVV}}\to \CH(\bar\M_{0,\mathcal K})\to \CH(\M_{0,S})\to 0.\]
   This gives an $R_{\mathcal K}$-module presentation for $\CH(\bar\M_{0,\mathcal K})$ as the free $R_{\mathcal K}$-module on $1$ and $\delta$ for codimension one graphs $\Delta$, modulo the relations
  \begin{align*}
         \tilde \gamma_1\cdot\ldots\cdot\tilde\gamma_d\cdot \delta,&\text{ for all  $\{\Gamma_i\}_{i=1}^d,\Delta$ with }\Gamma_1\wedge\ldots\wedge\Gamma_d\wedge\Delta=\emptyset\\
         \tilde \gamma_1\delta-\tilde \gamma_2 \delta, &\text{ for all  $\Gamma_1,\Gamma_2,\Delta$ with }\Gamma_1\wedge\Delta=\Gamma_2\wedge\Delta\\
        \sum_{\substack{\text{codim $1$ }\Delta\\ ij\Delta k\ell}} \delta-\sum_{\substack{\text{codim $1$ }\Delta\\ ik\Delta j\ell}} \delta,&\text{ for distinct }i,j,k,\ell\in S\\
        \tilde{\delta}'\delta-\tilde{\delta}\delta', &\text{ for all  $\Delta,\Delta'$}\\
        \tilde{\delta}\cdot 1-\delta,&\text{ for all $\Delta$.}
    \end{align*}
    We see that we only need to include $1$ as a generator. Writing all relations in terms of $1$, because of the defining relations of $R_{\mathcal K}$, they all vanish. Therefore, $\CH(\bar\M_{0,\mathcal K})$ is a free $R_{\mathcal K}$-module on the generator $1$. Hence, $h_{\mathcal K}$ is an isomorphism.
\end{proof}

\section{Examples}

\subsection{$\bar\M_{0,n}$}\label{5.1}

It is not immediately clear that Theorem~\ref{main2} recovers Keel's presentation: there are extra relations:
\begin{enumerate}
    \item $\delta_1\cdot\ldots\cdot \delta_{d}=0\text{ if }  \Delta_1\wedge\dots\wedge\Delta_d=\emptyset \text{ for } d>2$ and
 \item $\delta_1\delta_2=\delta_1\delta_3$ if $\Delta_1\wedge\Delta_2=\Delta_1\wedge\Delta_3$.
\end{enumerate}
For the second set of relations, we note that when $\Delta_1$ has two vertices, $\Delta_1\wedge\Delta_2=\Delta_1\wedge \Delta_3\implies \Delta_2=\Delta_3$; this was used in the proof of Proposition~\ref{kernel}. Hence, we need not include this second set of new relations.

The following proposition implies that the first set of relations is also superfluous:
\begin{prop}
    Let $\mathcal K$ be a simplicial complex on a set $S$ which is at least triparted. Suppose
    \begin{itemize}
        \item $\Delta_1\wedge\dots\wedge\Delta_d=\emptyset$,
        \item $\Delta_2\wedge\dots\wedge\Delta_d\neq \emptyset$, and 
        \item $\Delta_1$ has two vertices.
    \end{itemize}
    Then there exists $i\in \{2,\dots,d\}$ so that $\Delta_1\wedge\Delta_i=\emptyset$.
\end{prop}
\begin{proof}
We reduce to the case where all $\Delta_{i}$ have two vertices. Let $A=\{i\in [d]| \Delta_i$ has one vertex$\}$. If there is some $i\in A$ with $\Delta_i$ colliding two markings which are attached to different vertices of $\Delta_1$, then we are done because $\Delta_1\wedge \Delta_i=\emptyset$. Hence, we may assume otherwise. Then, consider $\bar\M_{\Lambda}:=\bigcap_{i\in A} \bar\M_{\Delta_i}$. Then $\bar\M_{\Lambda}=\bar\M_{0,\mathcal K'}$, where $\mathcal K'$ is a simplicial complex on a quotient of $S$, by Proposition~\ref{special} and Corollary~\ref{tautologicalcalculus}. We intersect everything with $\bar\M_{0,\mathcal K'}$. For $i=1$, we have that $\bar\M_{\Delta_1}\cap \bar\M_{\Lambda}$ is nonempty as curves with two components with the required distribution and collisions of marked points live in $\bar\M_{0,\mathcal K}$, and hence in this intersection. Moreover, $\Delta_2\wedge\dots\wedge\Delta_d\neq \emptyset$ implies that $\bar\M_{\Delta_i}\cap \bar\M_\Lambda\neq \emptyset$. For $i\in [d]\setminus A$, let $\bar\M_{\Delta_i'}:=\bar\M_{\Delta_i}\cap \bar\M_{\Lambda}\subseteq \bar\M_{\Lambda}=\bar\M_{0,\mathcal K'}$. If we had some $i\in \{2,\dots,d\}\setminus A$ with $\bar\M_{\Delta_1'}\cap \bar\M_{\Delta_i'}=\emptyset$, then each of the two vertices of $\Delta_i'$ has markings from both vertices of $\Delta_1'$. Then, the same is true for $\Delta_i$ and $\Delta_1$, so $\bar\M_{\Delta_1}\cap \bar\M_{\Delta_i}=\emptyset$. Thus, we are indeed reduced to the case where all $\Delta_i$ have two vertices.

Define $\mathcal K$-stable graphs $\Psi$ and $\Psi'$ so that $\bar\M_{\Psi}=\bar\M_{\Delta_2}\cap\dots\cap \bar\M_{\Delta_d}$ and $\bar\M_{\Psi'}=\bar\M_{\Delta_3}\cap\dots\cap \bar\M_{\Delta_d}$. We may assume that the codimension of $\bar\M_{\Psi}$ is $d-1$. Then we have $\Delta_1\wedge \Psi=\emptyset$, so there is an edge $e\in E(\Psi)$ whose induced partition of $S$ is incompatible with that of $\Delta_1$. If $\Delta_1\wedge \Psi'=\emptyset$, we are done by induction, so suppose otherwise. We have that $\Psi'$ is obtained from $\Psi$ by contracting some edge $e'\in E(\Psi)$, and the induced partition of $S$ is that of $\Delta_2$. Because $\Delta_1\wedge \Psi'\neq \emptyset$, all partitions of $S$ coming from $E(\Psi')$ are compatible with the partition induced by $\Delta_1$. Hence, we have $e=e'$, and so the partitions of $\Delta_1$ and $\Delta_2$ are incompatible. Thus, $\Delta_1\wedge \Delta_2=\emptyset$.
\end{proof}

We can note that the above gives that some of the relations are redundant not just for $\bar\M_{0,n}$. At the expense of more notation and a longer list of relations, we can give another version of Theorem~\ref{main2} that gets rid of these redundancies. 

Let $\mathcal K$ be a simplicial complex on a set $S$. Given a partition $\mathfrak p$ of $S$ into two parts with neither part in $\mathcal K$, let $\Phi_{\mathfrak p}$ denote the codimension one $\mathcal K$-stable graph whose induced partition of $S$ equals $\mathfrak p$. Given $\{i,j\}\in \mathcal K$, let $\Theta_{ij}$ be the codimension one graph with $i,j$ at the same half edge.  Recall that two partitions $\mathfrak p$ and $\mathfrak q$ of $S$ are said to be incompatible if there is no containment between the elements of $\mathfrak p$ and the elements of $\mathfrak q$.
\begin{thm}\label{main3}
   The Chow ring of $\bar\M_{0,\mathcal K}$ is generated by the classes $\{\delta\}_\Delta=\{\phi_{\mathfrak p}\}_{\mathfrak p}\cup \{\theta_{st}\}_{st}$ for codimension one $\mathcal K$-stable graphs $\Delta$, subject to only the relations
    \begin{align*}
        \phi_{\mathfrak p}\phi_{\mathfrak q}=0 & \text{ if $\mathfrak p$ and $\mathfrak q$ are incompatible},\\
        \phi_{\mathfrak p}\theta_{ij}=0 & \text{ if $i,j$ are in different parts of $\mathfrak p$,}\\
        \theta_{i_1i_2}\cdot\ldots\cdot \theta_{i_{d}i_{d+1}}=0 & \text{ if $\{i_1,\dots,i_{d+1}\}\notin \mathcal K$,}\\
        \theta_{ij}\theta_{ik}=\theta_{ij}\theta_{jk} & \text{ if $\{i,j,k\}\in\mathcal K$}\text{, and}\\
        \sum_{\substack{\text{codim $1$ }\Delta\\ ij\Delta k\ell}}{\delta}=\sum_{\substack{\text{codim $1$ }\Delta\\ ik\Delta j\ell}} {\delta} & \text{ for distinct } i,j,k,\ell\in S.
    \end{align*}
\end{thm}

The one thing that still needs explanation for obtaining Theorem~\ref{main3} from Theorem~\ref{main2} is the following:

Suppose we have stable graphs
\[\Theta_{i_1j_1},\dots,\Theta_{i_dj_{d}},\Theta_{i_1'j_1'},\dots,\Theta_{i_e'j_{e}'}\] 
such that
\begin{itemize}
    \item $\Theta_{i_1j_1}\wedge\dots\wedge\Theta_{i_dj_d}\neq \emptyset$,
    \item $\Theta_{i_1'j_1'}\wedge\dots\wedge\Theta_{i_e'j_{e}'}\neq \emptyset$, and 
    \item $\{i_1,j_1,\dots,i_d,j_d\}\cap \{i_1',j_1',\dots,i_e',j_{e}'\}=\emptyset$.
\end{itemize}
Then $\Theta_{i_1j_1}\wedge\dots\wedge\Theta_{i_dj_d}\wedge\Theta_{i_1'j_1'}\wedge\dots\wedge\Theta_{i_e'j_{e}'}\neq \emptyset$. This is because the collided markings are disjoint: writing down a curve with the desired collisions is necessarily $\mathcal K$-stable. Hence, if a collection $\{\bar\M_{\Theta_{i_sj_s}}\}_{s\in [d]}$ has empty intersection, one can find a subcollection $\{\bar\M_{\Theta_{i_sj_s}}\}_{s \in A}$ with empty intersection so that for each $s\in A$, either $i_s$ or $j_s$ is in $\{i_t,j_t\}$ for some $t\neq s\in A$. Then, one can use the relations $\theta_{ij}\theta_{ik}=\theta_{ij}\theta_{jk}$ to get the relation of the type in Theorem~\ref{main3}.

\subsection{$(\mathbb P^1)^{n-3}$}
The compactification $\M_{0,n}\subseteq \mathbb A^{n-3}\subseteq (\mathbb P^1)^{n-3}$ is isomorphic to $\bar\M_{0,\mathcal K_0}$, where $\mathcal K_0$ is the simplicial complex
\[\mathcal K_0:=\{A\subseteq [n]| \#(A\cap \{1,2,3\})\leq 1\}.\]
Keel's proof of \cite[Theorem 1]{Keel} works, in some sense, by realizing $\bar\M_{0,n}$ as a blow-up of this $(\mathbb P^1)^{n-3}$ at smooth centers, and the fact that blowing up at smooth centers predictably affects the Chow ring.

We use Theorem~\ref{main3} to compute $\CH(\bar\M_{0,\mathcal K_0})=\CH((\mathbb P^1)^{n-3})$. Note that for all subsets $A\subseteq [n]$, either $A\in\mathcal K_0$ or $[n]\setminus A\in \mathcal K_0$. Thus, the generators of $\CH(\bar\M_{0,\mathcal K_0})$ are given by $\theta_{ij}$ for $\{i,j\}$ not contained in $\{1,2,3\}$. The WDVV relations give
\[\theta_{ij}+\theta_{k\ell}=\theta_{ik}+\theta_{j\ell}\]
for $i,j,k,\ell\geq 4$,  
\begin{align*}
    \theta_{ij}&=\theta_{1i}+\theta_{2j}\\
    &=\theta_{1i}+\theta_{3j}\\
    &=\theta_{2i}+\theta_{3j}
\end{align*}
for $i,j\geq 4$, and
\[\theta_{1i}=\theta_{2i}=\theta_{3i}\]
for $i\geq 4$. If we set $h_i:=\theta_{1i}$ for $i\in [n]\setminus \{1,2,3\}$, the above says that the $h_i$ generate and there are no relations coming from WDVV. Any set of markings is allowed to collide unless it contains two of $\{1,2,3\}$, so the only relations coming from products of $\theta_{ij}$ are $\theta_{1i}\theta_{2i},\theta_{1i}\theta_{3i}$, and $\theta_{2i}\theta_{3i}$. Hence, we get the relations $h_i^2=0$. Finally, in terms of the $h_i$, the relations $\theta_{ij}\theta_{ik}=\theta_{ij}\theta_{jk}$ say $(h_i+h_j)(h_i+h_k)=(h_i+h_j)(h_j+h_k)$. Using the previous degree $2$ relations, we see both sides are equal to $h_ih_k+h_ih_j+h_jh_k$, hence we get no new relations from this. Thus, the presentation from Theorem~\ref{main3} simplifies to
\[\CH(\bar\M_{0,\mathcal K_0})=\frac{\mathbb Z[h_4,\dots,h_n]}{(h_4^2,\dots,h_n^2)},\]
which is, of course, the Chow ring of $(\mathbb P^1)^{n-3}$.

\subsection{Necessity of New Relations}\label{necessity}
Here, we provide examples that show that the relations in $\CH(\bar\M_{0,\mathcal K})$ which are not generalizations of the relations of Keel, namely
\begin{enumerate}
    \item $\delta_1\cdot\ldots\cdot \delta_{d}=0\text{ if }  \{\bar\M_{\Delta_i}\}_{i=1}^d\text{ has empty intersection for}$ $d>2$ and
 \item $\delta_1\delta_2=\delta_1\delta_3$ if $\Delta_1\wedge\Delta_2=\Delta_1\wedge\Delta_3$
\end{enumerate}
are necessary in certain cases.

For the first set of relations, let $S=[5]$ and $\mathcal K$ be the simplicial complex on $S$ generated by the cardinality $3$ subsets of $[4]$. Then $\mathcal K$ is at least triparted. Let $R'_{\mathcal K}$ be the ring with generators $\tilde \delta$ as before, but with only the relations $\tilde\delta\cdot \tilde\delta'=0$ and the WDVV relations. Using a computer algebra system, one can verify that 
\begin{enumerate}
    \item the relations $\tilde \delta_1\cdot\tilde\delta_2\cdot \tilde \delta_3=0$ for $\Delta_1\wedge\Delta_2\wedge\Delta_3=\emptyset$ do not hold on $R'_{\mathcal K}$, and
    \item the relations $\tilde\delta_1\tilde\delta_2=\tilde\delta_1\tilde\delta_3$ do hold on $R'_{\mathcal K}$.
\end{enumerate}

For the second set above, let $S=[6]$ and $\mathcal K$ be the simplicial complex generated by edges between any two vertices and all triangles containing $1$. Then $\mathcal K$ is at least triparted. Let $R'_{\mathcal K}$ be the ring with generators $\tilde \delta$ as before, but with only the relations $\tilde\delta\cdot \tilde\delta'=0$ and the WDVV relations. Using a computer algebra system, one can verify that 
\begin{enumerate}
    \item the relations $\tilde \delta_1\cdot\ldots\cdot \tilde \delta_d=0$ for $\Delta_1\wedge\dots\wedge\Delta_d=\emptyset$ hold on $R'_{\mathcal K}$, and
    \item the relations $\tilde \delta_1\tilde \delta_2=\tilde \delta_1\tilde \delta_3$ do not hold on $R'_{\mathcal K}$.
\end{enumerate}

In certain circumstances other than $\bar\M_{0,n}$, we can conclude that we do not need certain relations. The following proposition gives a criterion for when the second set of relations above is unnecessary. These conditions are not needed for the second set of relations to be redundant, as the example of $(\mathbb P^1)^{n-3}$ shows. We use notation as in Theorem~\ref{main3}.
\begin{prop}\label{fewer}
    Suppose for all $i,j\in S$, there exists some $\ell$ so that $\{i,j,\ell\}\notin \mathcal K$. Then the relations $\theta_{ij}\theta_{ik}=\theta_{ij}\theta_{jk}$ for $\{i,j,k\}\in \mathcal K$ are implied by the WDVV and the $d=2$ disjoint divisor relations. 
\end{prop}
\begin{proof}
    Suppose $\{i,j,k\}\in \mathcal K$. Choose $\ell$ so that $\{i,j,\ell\}\notin \mathcal K$. Multiply the WDVV relation associated to $i,k,\ell,j$ by $\theta_{ij}$. Because $i$ and $j$ are on opposite sides of $\Delta$, the only surviving terms are
    \[\theta_{ij}\theta_{ik}+\theta_{ij}\theta_{j\ell}=\theta_{ij}\theta_{jk}+\theta_{ij}\theta_{i\ell}.\]
    Because $\{i,j,\ell\}\notin\mathcal K$, the second terms from each side are $0$, giving the desired relation.
\end{proof}

The author does not know a similar criterion for the first set of relations.

\subsection{Heavy/Light}
Pick $2\leq m\leq n$ with $n\geq 3$. Let $\mathcal K_{m,n}$ be the simplicial complex generated by the subset $\{m+1,\dots,n\}$. Then $\bar\M_{0,\mathcal K_{m,n}}$ is always a Hassett space, called a \emph{heavy/light} Hassett space. The $\{1,\dots,m\}$ are ``heavy'', in that they are not permitted to collide with anything, and the $\{m+1,\dots,n\}$ are ``light'' in that any number of them may collide with each other. 

The Chow rings of these spaces were computed in \cite{KKL}. Their main result is equivalent to
\begin{thm}\label{KKL} \!\!\!\!\emph{\cite[Theorem 1.3]{KKL}}
    The Chow ring of $\bar\M_{0,\mathcal K_{m,n}}$ is generated by classes of boundary divisors $\phi$ where $\Phi$ has two vertices, subject to only the relations
    \begin{align*}
         \phi\cdot \phi'=0 &\text{ if }\Phi\wedge \Phi'=\emptyset\text{, and}\\ 
        \sum_{\substack{\text{codim $1$ }\Phi\\\#V(\Phi)=2\\ ij\Phi k\ell}} \phi=\sum_{\substack{\text{codim $1$ }\Phi\\\#V(\Phi)=2\\  ik\Phi j\ell}} \phi&\text{ for distinct }i,j,k,\ell\in [n]\text{ with }i,\ell\leq m.
    \end{align*}
\end{thm}
\begin{rmk}
    The linear relations given in \cite[Theorem 1.3]{KKL} are
    \[\sum_{\Phi\in T(\{i,j\},\{k,\ell\})} \phi=\sum_{\Phi\in T(\{k,\ell\},\{i,j\})} \phi,\]
    for $\{i,j\},\{k,\ell\}\subseteq [n]\setminus \{1\}$ subsets of cardinality $2$ with $i,k\leq m$, where 
    \[T(\{i,j\},\{k,\ell\}):=\{\Phi| ij\Phi 1k\ell  \text{ or }ijk\Phi1\ell \text{ or }ij\ell \Phi 1k\}.\]
    These linear relations are equivalent to the linear relations listed in Theorem~\ref{KKL}. First, we note that there are two kinds of relations of the above type, depending on whether $\{i,j\}$ and $\{k,\ell\}$ are disjoint or if they share an element. It can be checked that the above relation for the subsets $\{i,j\},\{i,k\}$ is exactly the WDVV relation coming from the indices $i,j,k,1$. Next, one also has that the relation coming from the disjoint subsets $\{i,j\}$ and $\{k,\ell\}$ is the sum of the WDVV relations for $i,j,\ell,1$ and $\ell,i,k,1$. Hence, the linear relations in \cite[Theorem 1.3]{KKL} are implied by the ones above. Conversely, the sum of the WDVV relations associated to $i,j,k,1$ and $\ell,k,j,1$ is the WDVV relation for $i,j,k,\ell$, hence the relations are equivalent. 
\end{rmk}

We show that this result follows from the results proven above in this paper.
\begin{proof}[Proof of Theorem~\ref{KKL}]
    We use Theorem~\ref{main3} as a baseline. Note that there are no relations of the form 
    \[\theta_{i_1i_2}\cdot\ldots\cdot \theta_{i_{d-1}i_d}=0  \text{ if $\{i_1,\dots,i_d\}\notin \mathcal K_{m,n}$,}\]
    since every subset of $[n]\setminus [m]$ is in $\mathcal K_{m,n}$ and only elements in $[n]\setminus [m]$ can appear in subscripts of a $\theta$. Next, we do not need the relations $\theta_{ij}\theta_{ik}=\theta_{ij}\theta_{jk}$ by Proposition~\ref{fewer}. Hence, the only relations we need are the product of two disjoint divisors and the WDVV relations. It suffices to remove the generators $\theta_{ij}$.

    Fix $i,j>m$ and choose any distinct $a,b \leq m$. The WDVV relation associated to $i,j,a,b$ gives
    \begin{equation}\label{theta}
        \theta_{ij}=\sum_{ia\Phi jb} \phi- \sum_{\substack{ij\Phi ab\\ \Phi\neq \Theta_{ij}}} \phi,
    \end{equation}
and all $\Phi$ appearing in the above have two vertices. We get different expressions for $\theta_{ij}$ by choosing different $a,b\leq m$; the difference in expression between choosing $a,b$ and $a,b'$ is given by the WDVV relation associated to $a,b',b,j$, so the expression of $\theta_{ij}$ in terms of the generators is well defined up to our WDVV relations in the statement.  

We need to check that the WDVV relations associated to $i,j,k,\ell$ involving some $\theta$ do not impose more relations on the $\phi$ classes than the WDVV relations mentioned in the theorem statement. If we have $i,j>m\geq k,\ell$, then the WDVV relation is just our above expression of $\theta_{ij}$ in terms of the generators. If $i,j,k>m\geq \ell$, one gets an expression involving $\theta_{ij}$ and $\theta_{ik}$. By expressing these as in \eqref{theta} with the choice $a=\ell$, one gets that this relation is the WDVV relation associated to $j,b,\ell,k$. Finally, if all $i,j,k,\ell>m$, then one can use the fact that the sum of the WDVV relations associated to $i,j,k,1$ and $\ell,k,j,1$ is the WDVV relation associated to $i,j,k,\ell$ to see that no new linear relations on the $\phi$ are imposed.

The last relations to check are those of the form $\phi_{\mathfrak p}\theta_{ij}$ when $\mathfrak p$ has $i$ and $j$ in different parts. If we choose $a,b\leq m$ with $a$ in the same part of $\mathfrak p$ as $j$ and $b$ in the same part of $\mathfrak p$ as $i$, then it is clear that $\Phi_{\mathfrak p}\wedge \Phi=\emptyset$ for all $\Phi$ appearing in the expression \eqref{theta} for $\theta_{ij}$. 
\end{proof}

\bibliographystyle{amsalpha}
\bibliography{refs.bib}

@inproceedings{BS1,
  title={Chow rings of stacks of prestable curves I},
  author={Bae, Younghan and Schmitt, Johannes and Skowera, Jonathan},
  booktitle={Forum of Mathematics, Sigma},
  volume={10},
  pages={e28},
  year={2022},
  organization={Cambridge University Press}
}

@article{BS2,
  title={Chow rings of stacks of prestable curves II},
  author={Bae, Younghan and Schmitt, Johannes},
  journal={Journal f{\"u}r die reine und angewandte Mathematik (Crelles Journal)},
  volume={2023},
  number={800},
  pages={55--106},
  year={2023},
  publisher={De Gruyter}
}

@article{BDL2,
  title={Integral Chow rings of modular compactifications of {${\mathcal{M}}_{1,n\leq 6}$}},
  author={Battistella, Luca and Di Lorenzo, Andrea},
  journal={arXiv preprint arXiv:2505.04587},
  year={2025}
}

@article{BB,
  title={On compactifications of {$\mathcal{M}_{g,n}$} with colliding markings},
  author={Blankers, Vance and Bozlee, Sebastian},
  journal={Selecta Mathematica},
  volume={30},
  number={5},
  pages={104},
  year={2024},
  publisher={Springer}
}

@article{CL2,
  title={On the Chow and cohomology rings of moduli spaces of stable curves},
  author={Canning, Samir and Larson, Hannah},
  journal={Journal of the European Mathematical Society},
  year={2024}
}

@article {Ceyhan,
    AUTHOR = {Ceyhan, \"Ozg\"ur},
     TITLE = {Chow groups of the moduli spaces of weighted pointed stable
              curves of genus zero},
   JOURNAL = {Adv. Math.},
  FJOURNAL = {Advances in Mathematics},
    VOLUME = {221},
      YEAR = {2009},
    NUMBER = {6},
     PAGES = {1964--1978},
      ISSN = {0001-8708,1090-2082},
   MRCLASS = {14H10 (14C15 14D22)},
  MRNUMBER = {2522832},
MRREVIEWER = {Dmitry\ Kerner},
       DOI = {10.1016/j.aim.2009.03.011},
       URL = {https://doi.org/10.1016/j.aim.2009.03.011},
}

@article{DLPV,
  title={Stable cuspidal curves and the integral Chow ring of {$\bar{\mathscr{M}}_{2,1}$}},
  author={Di Lorenzo, Andrea and Pernice, Michele and Vistoli, Angelo},
  journal={Geometry \& Topology},
  volume={28},
  number={6},
  pages={2915--2970},
  year={2024},
  publisher={Mathematical Sciences Publishers}
}

@book{Fulton,
  title={Intersection theory},
  author={Fulton, William},
  volume={2},
  year={2013},
  publisher={Springer Science \& Business Media}
}

@article{GP,
  title={Constructions of nontautological classes of moduli spaces of curves},
  author={Graber, Thomas and Pandharipande, Rahul},
  journal={Michigan Mathematical Journal},
  volume={51},
  number={1},
  pages={93--110},
  year={2003},
  publisher={University of Michigan, Department of Mathematics}
}

@article{Hassett,
  title={Moduli spaces of weighted pointed stable curves},
  author={Hassett, Brendan},
  journal={Advances in Mathematics},
  volume={173},
  number={2},
  pages={316--352},
  year={2003},
  publisher={Elsevier}
}

@article{KKL,
  title={Chow rings of heavy/light {H}assett spaces via tropical geometry},
  author={Kannan, Siddarth and Karp, Dagan and Li, Shiyue},
  journal={Journal of Combinatorial Theory, Series A},
  volume={178},
  pages={105348},
  year={2021},
  publisher={Elsevier}
}

@article{Keel,
  title={Intersection theory of moduli space of stable n-pointed curves of genus zero},
  author={Keel, Sean},
  journal={Transactions of the American Mathematical Society},
  pages={545--574},
  year={1992},
  publisher={JSTOR}
}

@article{KM94,
  title={Gromov-Witten classes, quantum cohomology, and enumerative geometry},
  author={Kontsevich, Maxim and Manin, Yu},
  journal={Communications in Mathematical Physics},
  volume={164},
  number={3},
  pages={525--562},
  year={1994},
  publisher={Springer}
}

@article{KM96,
  title={Quantum cohomology of a product},
  author={Kontsevich, M and Manin, Yu},
  journal={Invent. math},
  volume={124},
  pages={313--339},
  year={1996}
}

@article {ELarson,
    AUTHOR = {Larson, Eric},
     TITLE = {The integral {C}how ring of {$\overline M_2$}},
   JOURNAL = {Algebr. Geom.},
  FJOURNAL = {Algebraic Geometry},
    VOLUME = {8},
      YEAR = {2021},
    NUMBER = {3},
     PAGES = {286--318},
      ISSN = {2313-1691,2214-2584},
   MRCLASS = {14H10 (14C25)},
  MRNUMBER = {4206438},
MRREVIEWER = {Zijian\ Zhou},
       DOI = {10.14231/ag-2021-007},
       URL = {https://doi.org/10.14231/ag-2021-007},
}

@article{MSvAX,
  title={Birational contractions of {$\bar{M}_{0,n}$} and combinatorics of extremal assignments},
  author={Moon, Han-Bom and Summers, Charles and von Albade, James and Xie, Ranze},
  journal={Journal of Algebraic Combinatorics},
  volume={47},
  number={1},
  pages={51--90},
  year={2018},
  publisher={Springer}
}

@article{Me1,
  title={The first higher Chow groups of {$\bar{\mathcal{M}}_{1,n}$} for {$n\leq 4$}},
  author={Newman, William C},
  journal={arXiv preprint arXiv:2508.20264},
  year={2025}
}

@article{Smyth,
  title={Towards a classification of modular compactifications of {$\mathcal{M}_{g,n}$}},
  author={Smyth, David Ishii},
  journal={Inventiones mathematicae},
  volume={192},
  number={2},
  pages={459--503},
  year={2013},
  publisher={Springer}
}

@misc{stacks-project,
  author       = {The {Stacks project authors}},
  title        = {The Stacks project},
  howpublished = {\url{https://stacks.math.columbia.edu}},
  year         = {2026},
}

@article{Totaro,
  title={The motive of a classifying space},
  author={Totaro, Burt},
  journal={Geometry \& Topology},
  volume={20},
  number={4},
  pages={2079--2133},
  year={2016},
  publisher={Mathematical Sciences Publishers}
}

\end{document}